\documentclass[10pt]{amsart}
\usepackage[a4paper,marginratio=1:1]{geometry}
\usepackage[initials,non-sorted-cites]{amsrefs}
\usepackage{mybibspec}
\usepackage{tensor,braket,bm,fouridx}

\newtheorem{thm}{Theorem}[section]
\newtheorem{prop}[thm]{Proposition}
\newtheorem{lem}[thm]{Lemma}

\theoremstyle{definition}
\newtheorem{dfn}[thm]{Definition}
\theoremstyle{remark}

\newcommand{\abs}[1]{\lvert#1\rvert}
\newcommand{\norm}[1]{\|#1\|}

\newcommand{\conj}[1]{\smash{\overline{#1}}}
\newcommand{\bdry}{\partial}
\newcommand{\partialbar}{{\smash{\conj{\partial}}}}
\newcommand{\spacewedge}{{\mathord{\wedge}}}
\DeclareMathOperator{\Ric}{Ric}
\DeclareMathOperator{\dom}{dom}
\DeclareMathOperator{\im}{im}
\DeclareMathOperator{\supp}{supp}
\DeclareMathOperator{\tr}{tr}

\let\div\relax
\DeclareMathOperator{\div}{div}
\DeclareMathOperator{\re}{Re}

\DeclareMathOperator{\rank}{rank}
\DeclareMathOperator{\Diff}{Diff}
\DeclareMathOperator{\dist}{dist}
\DeclareMathOperator{\Span}{span}
\DeclareMathOperator{\Hom}{Hom}
\DeclareMathOperator{\End}{End}

\newcommand{\PSU}{\mathit{PSU}}
\newcommand{\Sp}{\mathit{Sp}}

\numberwithin{equation}{section}

\title[Deformation of Einstein metrics and $L^2$ cohomology]{Deformation of Einstein metrics and $L^2$ cohomology on strictly pseudoconvex domains}
\author{Yoshihiko Matsumoto}
\thanks{Partially supported by Grant-in-Aid for JSPS Fellows (14J11754).}
\address{Department of Mathematics, Graduate School of Science, Osaka University,
	1-1 Machikaneyama-cho, Toyonaka, Osaka 560-0043, Japan}
\email{matsumoto@math.sci.osaka-u.ac.jp}
\subjclass[2010]{Primary 53C25; Secondary 32L20, 32Q20, 32T15, 32V05.}
\keywords{Complete Einstein metrics; $L^2$ cohomology; asymptotically complex hyperbolic metrics;
partially integrable CR structures}

\begin{document}

\begin{abstract}
	We construct new complete Einstein metrics on smoothly bounded strictly pseudoconvex domains
	in Stein manifolds.
	This is done by deforming the K\"ahler-Einstein metric of Cheng and Yau,
	the approach that generalizes the works of Roth and Biquard
	on the deformations of the complex hyperbolic metric on the unit ball.
	Recasting the problem into the question of vanishing of an $L^2$ cohomology
	and taking advantage of the asymptotic complex hyperbolicity of the Cheng--Yau metric,
	we establish the possibility of such a deformation when the dimension of the domain is
	larger than or equal to three.
\end{abstract}

\maketitle

\section*{Introduction}

Let $\Omega$ be a smoothly bounded strictly pseudoconvex domain in a Stein manifold
of dimension $n\ge 2$.
It is shown by Cheng and Yau \cite{Cheng-Yau-80} that then
$\Omega$ carries a complete K\"ahler-Einstein metric $g$
with negative scalar curvature, which is unique up to homothety.
This metric has attracted much interest in connection with CR geometry of the boundary $\bdry\Omega$.
Actually, the natural CR structure on $\bdry\Omega$, which is called the \emph{conformal infinity} of $g$
in our context, locally and asymptotically determines the metric $g$ up to a certain high order,
as first pointed out by Fefferman \cite{Fefferman-76} and further investigated by Graham \cite{Graham-87}.
It is a version of bulk-boundary correspondence, which is more extensively studied in the setting of
asymptotically \emph{real} hyperbolic metrics
(and in that of asymptotically anti-de Sitter metrics in physical context).

One can further generalize this complex bulk-boundary correspondence
toward so-called asymptotically complex hyperbolic (ACH) Einstein metrics and CR structures
that are not necessarily integrable,
which is the idea of Roth \cite{Roth-99-Thesis} and Biquard \cite{Biquard-00}.
Those that are admitted as conformal infinities in the new setting
are called partially integrable CR structures (or partially integrable almost CR structures) in the literature
(see, e.g., \cites{Cap-Schichl-00,Cap-Slovak-09,Matsumoto-16}).

What is shown in \cite{Roth-99-Thesis} and \cite{Biquard-00}
is the perturbative existence and uniqueness result on the ball
(the existence is treated by both, and the uniqueness is due to \cite{Biquard-00}):
For any partially integrable CR structure $J$ on the sphere $S^{2n-1}$ sufficiently close to the standard one,
there exists an Einstein ACH metric on the unit ball $B^n\subset\mathbb{C}^n$
close to the complex hyperbolic metric, which is ``locally unique'' up to diffeomorphism action,
whose conformal infinity is $J$.
This result is parallel to that of Graham and Lee~\cite{Graham-Lee-91} for asymptotically real hyperbolic (AH)
Einstein metrics.

In this paper, we take up the same perturbation problem on an \emph{arbitrary} bounded strictly
pseudoconvex domain $\Omega$.
Our seed metric is the Cheng--Yau metric, and by the \emph{natural} CR structure on $\bdry\Omega$
we mean the one induced by the complex structure of the ambient Stein manifold.
Then we can establish the following generalization of
the result of \cite{Roth-99-Thesis} and \cite{Biquard-00} provided the dimension $n$ is at least three.

\begin{thm}
	\label{thm:main}
	Let $\Omega$ be a smoothly bounded strictly pseudoconvex domain in a Stein manifold of dimension $n\ge 3$.
	Suppose that $J$ is a partially integrable CR structure on the boundary $\bdry\Omega$
	sufficiently close to the natural CR structure in the $C^{2,\alpha}$ topology.
	Then there exists an Einstein ACH metric $g$ on $\Omega$ with conformal infinity $J$.
	The metric $g$ is locally unique modulo the action of diffeomorphisms on $\Omega$
	inducing the identity on $\bdry\Omega$.
\end{thm}

Here we are implicitly assuming that the underlying contact structure of $J$ remains to be the natural one,
so that the $C^{2,\alpha}$-closeness between $J$ and the natural CR structure is well-defined.
As contact structures on closed manifolds are rigid (see Gray~\cite{Gray-59}), no generality is lost by this
assumption.

The uniqueness is precisely stated as follows using the weighted H\"older space
(see Section \ref{subsec:ACH-metrics} for details):
If $g$ is our metric and $\Hat{g}$ is another Einstein ACH metric
such that $\Hat{g}-g\in C^{2,\alpha}_\delta(S^2T^*\Omega)$ for some $\delta>0$
(which in particular implies that their conformal infinities are the same),
then there exists a diffeomorphism $\Phi\in\Diff(\Omega)$ that is at least continuous up to the boundary
and restricts to the identity on $\bdry\Omega$ for which $\Phi^*\Hat{g}=g$.
This uniqueness part is not a new result---it is discussed in \cite{Biquard-00}*{Proposition I.4.6}.

The framework of our proof of the existence is standard.
Namely, we apply the inverse function theorem to a mapping between Banach spaces considered in \cite{Biquard-00}.
Then the problem becomes to see if the linearization of the mapping in question is isomorphic.
This is relatively easy for the complex hyperbolic metric on the unit ball;
hence the result of \cite{Roth-99-Thesis} and \cite{Biquard-00}.
However, for general domains, it is far less obvious. Here lies the main issue of our discussion.

Let us get into some more details.
We consider the following ``Bianchi-gauged'' Einstein equation
(we need a gauge-fixing condition in order to get rid of the diffeomorphism invariance of the Einstein equation):
\begin{equation}
	\label{eq:gauged-Einstein-equation}
	\mathcal{E}_g(h):=\Ric(g+h)-\lambda(g+h)+\delta_{g+h}^*\mathcal{B}_g(h)=0,
	\qquad \mathcal{B}_g(h):=\delta_gh+\frac{1}{2}d\tr_gh.
\end{equation}
Here $g$ denotes some ACH metric, $h$ is a symmetric 2-tensor that is ``small'' compared to $g$
(so that $g+h$ remains to be a metric in particular),
$\delta_g$ is the divergence with respect to $g$, and $\delta_{g+h}^*$ is the formal adjoint
of $\delta_{g+h}$ with respect to $g+h$.
We need to show that the linearization $\mathcal{E}_g'$ of
\eqref{eq:gauged-Einstein-equation} associated to the Cheng--Yau metric
at $h=0$ is an isomorphism between certain weighted H\"older spaces of symmetric 2-tensors.
The upshot of the discussions of Roth and Biquard (see also Lee~\cite{Lee-06} for the AH case)
is that it follows once the vanishing of the $L^2$ kernel of $\mathcal{E}_g'$ is established.
Therefore, the $L^2$ kernel, denoted by $\ker_{(2)}\mathcal{E}_g'$,
is called the \emph{obstruction space} of Einstein deformation.

The operator $\mathcal{E}_g'$, the \emph{linearized gauged Einstein operator},
for an Einstein metric $g$ turns out to be the following,
where $\mathring{R}$ denotes the usual action of the curvature tensor on symmetric 2-tensors:
\begin{equation}
	\label{eq:linearized-Einstein-formula}
	\mathcal{E}_g'=\frac{1}{2}(\nabla^*\nabla-2\mathring{R}).
\end{equation}
This expression shows, in particular, that the $L^2$ kernel vanishes when $g$ has negative sectional curvature
everywhere (thus we get the result for the unit ball).
The problem is that we do not know if a general Cheng--Yau metric $g$ satisfies the condition $K_g<0$.
Nevertheless, at infinity any Cheng--Yau metric is asymptotic to the complex hyperbolic metric,
and this suggests that we can overcome the difficulty by reducing it to an analysis \emph{near infinity}
in a suitable way.

Our crucial idea is to use Koiso's observation~\cite{Koiso-83} to recast the problem in terms of
an $L^2$ cohomology.
Note first that any symmetric 2-tensor $\sigma$ can be decomposed into the sum of
Hermitian and anti-Hermitian parts: $\sigma=\sigma_H+\sigma_A$.
For a K\"ahler-Einstein metric $g$, this decomposition is respected by $\mathcal{E}_g'$,
as can be seen from \eqref{eq:linearized-Einstein-formula}.
Now we can identify $\sigma_A$ with a $(0,1)$-form $\alpha$ with values in the
holomorphic tangent bundle $T^{1,0}$ by the duality induced by the metric,
and under this identification, it turns out that $\mathcal{E}_g'\sigma_A$ corresponds to
$\frac{1}{2}\Delta_\partialbar\alpha$.
Then, a little bit of further consideration leads to the conclusion that the vanishing
of the obstruction space $\ker_{(2)}\mathcal{E}_g'$ follows from
\begin{equation}
	\label{eq:vanishing-harmonic}
	\mathcal{H}_{(2)}^{0,1}(\Omega;T^{1,0})=0,
\end{equation}
the vanishing of the space of $L^2$ harmonic $T^{1,0}$-valued
$(0,1)$-forms on $\Omega$ with respect to the Cheng--Yau metric.
By the de Rham--Hodge--Kodaira decomposition on noncompact manifolds, there is an isomorphism
\begin{equation*}
	\mathcal{H}_{(2)}^{0,1}(\Omega;T^{1,0})\cong H^{0,1}_{(2), \mathrm{red}}(\Omega;T^{1,0}),
\end{equation*}
where the right-hand side is the so-called \emph{reduced} $L^2$ cohomology.
Although $H^{0,1}_{(2), \mathrm{red}}(\Omega;T^{1,0})$ and
the usual $L^2$ cohomology $H_{(2)}^{0,1}(\Omega;T^{1,0})$ can be different in general,
it is a trivial fact that if $H_{(2)}^{0,1}$ vanishes then so does $H_{(2), \mathrm{red}}^{0,1}$.
Thus the vanishing of the $L^2$ cohomology $H_{(2)}^{0,1}$ implies that of $\mathcal{H}_{(2)}^{0,1}$.

A virtue of this interpretation of the problem using cohomology is that
we can make use of a certain long exact sequence (see Ohsawa \cite{Ohsawa-91})
that, combined with a classical result in several complex variables, makes it sufficient to show that
the $L^2$ cohomology \emph{on a neighborhood of the boundary} vanishes.
Then this final form of the problem is solved by methods in the classical $\partialbar$-Neumann problem,
namely by the Morrey--Kohn--H\"ormander equality
and a technique of H\"ormander related to his ``condition $Z(q)$,'' when $n\ge 4$.

We need to modify our argument a little when $n=3$:
The vanishing of $\ker_{(2)}\mathcal{E}_g'$ should be reduced to
that of some \emph{weighted} $L^2$ cohomology instead of the usual $L^2$ cohomology.
This is possible by using a result of \cite{Biquard-00} again,
which states that the elements of $\ker_{(2)}\mathcal{E}_g'$ on a general ACH-Einstein manifold
actually satisfy a stronger decay property at infinity than just being $L^2$.
The introduction of the weight improves our estimate enough to show the desired vanishing result.
Unfortunately, the case $n=2$ cannot be settled even if we use this additional technique.
This case remains unsolved so far.

The argument outlined above is presented in detail in the following way.
In Section 1, we summarize the definition of ACH metrics and the associated weighted H\"older spaces.
After that the Einstein deformation theory of Roth and Biquard is recalled,
and the improved decay for the elements of $\ker_{(2)}\mathcal{E}_g'$ is also explained here.
In Section 2, we review the $L^2$ cohomology on noncompact Hermitian manifolds,
and the problem is reduced to the vanishing of the $L^2$ cohomology near the boundary.
In Section 3, we establish the necessary estimate and complete the proof of Theorem~\ref{thm:main}.

We include an appendix, in which we give yet another proof of a result of
Donnelly and Fefferman~\cite{Donnelly-Fefferman-83} on $L^2$ harmonic scalar-valued differential forms
using the improved decay technique.
In fact, via this technique, the vanishing part of the Donnelly--Fefferman theorem
boils down to the vanishing of some weighted $L^2$ cohomologies,
which we can prove by the standard Bochner--Kodaira--Nakano formula.

Most of this work was carried out during the author's stay at \'Ecole normale sup\'erieure at Paris
in 2014--15. I wish to thank Olivier Biquard for hosting the visit and for giving a lot of insightful advice,
which substantially improved my comprehension of geometric analysis of ACH metrics.
I would also like to express my gratitude to the staff of ENS and Tokyo Institute of Technology,
especially to Lara Morise and Akiko Takagi, for various practical help during the stay.
I was benefited from discussions with Masanori Adachi and Tomoyuki Hisamoto in an early stage of
the project. Gilles Carron suggested the idea of using the approach employed for the $\conj{\partial}$-Neumann
problem; on this point Takeo Ohsawa helped me with clarifying my understanding of the relevant exact sequence.
Ohsawa also remarked that the vanishing of $H^{0,1}_c(\Omega,E)$
(see Section \ref{subsec:reduction-to-L2-cohomology}) immediately follows from the Oka--Cartan theorem.
Shin-ichi Matsumura pointed out a redundant argument that had existed in an earlier version of the manuscript.

\section{Einstein deformation theory of asymptotically complex hyperbolic metrics}
\label{sec:ACH-metrics}

\subsection{The Cheng--Yau metric}

Cheng and Yau considered the Calabi problem on noncompact complex manifolds in \cite{Cheng-Yau-80}.
In particular, they established the following theorem.

\begin{thm}[Cheng--Yau~\cite{Cheng-Yau-80}]
	Let $\Omega$ be a smoothly bounded strictly pseudoconvex domain in a Stein manifold of dimension $n\ge 2$.
	Then, for each fixed $\lambda<0$, there exists a unique complete K\"ahler metric $g$ on $\Omega$
	satisfying $\Ric(g)=\lambda g$.
\end{thm}

Such a metric is constructed by solving the complex Monge--Amp\`ere equation.
More specifically, normalizing the metric by setting $\lambda=-(n+1)$,
we can obtain such an Einstein metric in the form $g=\tensor{g}{_i_{\conj{j}}}dz^id\conj{z}^j$ with
\begin{equation}
	\label{eq:Cheng-Yau-metric}
	\tensor{g}{_i_{\conj{j}}}=\frac{\partial^2(-\log\varphi)}{\partial z^i\partial\conj{z}^j}
	+\tensor{h}{_i_{\conj{j}}},
\end{equation}
where $\varphi\in C^\infty(\Omega)\cap C^{n+1,\alpha}(\overline{\Omega})$ is a positive defining function
of $\Omega$ and $\tensor{h}{_i_{\conj{j}}}$ is a Hermitian symmetric form on the ambient Stein manifold $Y$;
when $Y=\mathbb{C}^n$ one can always take $\tensor{h}{_i_{\conj{j}}}=0$.
In what follows, by the Cheng--Yau metric, this one is always referred to.
The uniqueness of the metric~\cite{Cheng-Yau-80}*{Theorem 8.3} follows from
Yau's Schwarz Lemma for volume forms~\cite{Mok-Yau-83}*{Section 1}.

Precisely speaking, in~\cite{Cheng-Yau-80} the existence of $\varphi$ is
stated explicitly only for domains in $\mathbb{C}^n$ and in K\"ahler manifolds admitting metrics of
negative Ricci curvature.
We argue as follows to see that it extends to the case of domains in Stein manifolds
(see also van Coevering~\cite{vanCoevering-12}*{Section 3.2} discussing in a more general setting).
Let $\tensor*{r}{^0_i_{\overline{j}}}$ be the Ricci tensor of some fixed K\"ahler metric $g^0$ on $Y$.
If $\psi\in C^\infty(\overline{\Omega})$ is a positive defining function
such that $-\psi$ is strictly plurisubharmonic on $\overline{\Omega}$, then on $\Omega$,
$\tensor{g}{_i_{\conj{j}}}=\partial_i\partial_{\conj{j}}(-\log\psi)-(n+1)^{-1}\tensor*{r}{^0_i_{\overline{j}}}$
defines a K\"ahler metric at least near the boundary.
The Steinness of $Y$ implies that
one can modify $\psi$ so that $\tensor{g}{_i_{\conj{j}}}$ becomes positive definite on the whole $\Omega$.
By~\cite{Cheng-Yau-80}*{Theorem 4.4}, there exists a function $u\in C^\infty(\Omega)$ for which
$\tensor{g}{_i_{\conj{j}}}+\tensor{u}{_i_{\conj{j}}}$ is a metric that is quasi-equivalent to
$\tensor{g}{_i_{\conj{j}}}$ and satisfies
\begin{equation*}
	\frac{\det(\tensor{g}{_i_{\conj{j}}}+\tensor{u}{_i_{\conj{j}}})}{\det(\tensor{g}{_i_{\conj{j}}})}
	=e^{(n+1)u+F},\qquad
	F=\log\left(\frac{\det(g^0_{i{\overline{j}}})}{\psi^{n+1}\det(\tensor{g}{_i_{\conj{j}}})}\right).
\end{equation*}
Then it can be seen that $\tensor{g}{_i_{\conj{j}}}+\tensor{u}{_i_{\conj{j}}}$ is an
Einstein K\"ahler metric.

A direct computation shows that the holomorphic sectional curvature of the
metric~\eqref{eq:Cheng-Yau-metric} uniformly tends to $-2$ at infinity
(see~\cite{Cheng-Yau-80}*{Equation (1.22)} for the case $Y=\mathbb{C}^n$), i.e.,
\begin{equation}
	\label{eq:asymptotic-curvature-behavior}
	\tensor{R}{_i_{\conj{j}}_k_{\conj{l}}}=
	-(\tensor{g}{_i_{\conj{j}}}\tensor{g}{_k_{\conj{l}}}+\tensor{g}{_i_{\conj{l}}}\tensor{g}{_k_{\conj{j}}})
	+o(1)\qquad \text{at $\bdry\Omega$},
\end{equation}
where the notation $o(1)$ means that this term has pointwise norm (with respect to $g$) that
tends to zero uniformly at $\bdry\Omega$.

The Cheng--Yau metric $g$ has bounded geometry in the sense that
its injectivity radius $r_\mathrm{inj}$ is positive
and the curvature tensor, as well as its covariant derivative of arbitrary order, is bounded.
The positivity of $r_\mathrm{inj}$ can be verified by using, for example,
a criterion due to Ammann--Lauter--Nistor~\cite{Ammann-Lauter-Nistor-04}*{Proposition 4.19}.
The boundedness of $\nabla^mR$, $m\ge 0$, follows from the boundedness of
$\nabla^mu$, shown in the proof of~\cite{Cheng-Yau-80}*{Theorem 4.4};
alternatively, one can use the existence of an asymptotic expansion of $\varphi$ involving
logarithmic terms established by Lee--Melrose~\cite{Lee-Melrose-82}.

\subsection{ACH metrics}
\label{subsec:ACH-metrics}

The model of asymptotically complex hyperbolic metrics is given by the following metric
on $M\times(0,1)$, where $x$ is the coordinate of the second factor:
\begin{equation}
	\label{eq:ACH-model}
	g_0=\frac{1}{2}\left(4\frac{dx^2}{x^2}+\frac{\theta^2}{x^4}+\frac{L_\theta}{x^2}\right).
\end{equation}
Here $M$ is a closed strictly pseudoconvex partially integrable CR manifold of dimension $2n-1$.
Recall that an almost CR manifold $(M,H,J)$ is said to be \emph{partially integrable} if
\begin{equation*}
	[\Gamma(T^{1,0}M),\Gamma(T^{1,0}M)]\subset\Gamma(T^{1,0}M\oplus\conj{T^{1,0}M}),
\end{equation*}
where $T^{1,0}M\subset\mathbb{C}H$ is the $i$-eigenbundle of $J\in\Gamma(\End(H))$.
An equivalent condition is that
\begin{equation*}
	L_\theta(\xi,\eta):=d\theta(\xi,J\eta),\qquad\xi,\,\eta\in H
\end{equation*}
is a symmetric form on $H$ for some (hence for any) 1-form $\theta$ for which $\ker\theta=H$.
A partially integrable almost CR manifold, or hereafter a \emph{partially integrable CR manifold},
is called \emph{strictly pseudoconvex} if $L_\theta$ has definite signature.
In this case, we always take $\theta$ such that the Levi form $L_\theta$ is positive definite,
which we call a \emph{pseudohermitian structure}.
(Note that the partial integrability is a pointwise condition, whereas the integrability is a condition
given by a differential equation.
For a fixed $\theta$, giving a strictly pseudoconvex partially integrable almost CR structure
for which $\theta$ is a pseudohermitian structure
amounts to giving a reduction of the structure group from $\Sp(2n-2,\mathbb{R})$ to $U(n-1)$,
and hence, such almost CR structures are exactly the elements of
the set of smooth sections of a certain fiber bundle with fibers diffeomorphic to $\Sp(2n-2,\mathbb{R})/U(n-1)$.)

Using the model metric \eqref{eq:ACH-model},
we define as follows (see also Biquard~\cite{Biquard-00},
Biquard--Mazzeo~\cites{Biquard-Mazzeo-06,Biquard-Mazzeo-11}).

\begin{dfn}
	\label{dfn:ACH-metrics}
	Let $X$ be a noncompact smooth manifold of real dimension $2n$, where $n\ge 2$,
	which compactifies into a smooth manifold-with-boundary $\overline{X}$.
	Then a Riemannian metric $g$ on $X$ is called \emph{asymptotically complex hyperbolic} (or ACH for short)
	if there exists a diffeomorphism between a neighborhood of $M:=\bdry\overline{X}$ in $\overline{X}$ and
	$M\times[0,\varepsilon)$ under which
	there is some strictly pseudoconvex partially integrable CR structure $(H,J)$ and a pseudohermitian
	structure $\theta$ on $M$ such that
	\begin{equation}
		\label{eq:ACH-definition}
		g=g_0+k,\qquad k\in C^{2,\alpha}_\delta(S^2T^*X)
	\end{equation}
	for some $\delta>0$, where $g_0$ is the model metric \eqref{eq:ACH-model} for $(M,H,J,\theta)$.
	The partially integrable CR structure $(H,J)$, or $J$, is called the \emph{conformal infinity} of $g$.
\end{dfn}

Here the space
\begin{equation*}
	C^{2,\alpha}_\delta(S^2T^*X):=x^\delta C^{2,\alpha}(S^2T^*X)
\end{equation*}
is the weighted H\"older space of symmetric 2-tensors on $X$ with respect to $g_0$.
For $0<\alpha<1$, the $C^{\alpha}$ H\"older norm of a section $s$ of a tensor bundle $E$ is defined by
\begin{equation}
	\label{eq:Holder-norm}
	\norm{s}_{C^{\alpha}}
	:=\sup\abs{s}
	+\sup_{\dist(p,q)<r_\mathrm{inj}}\frac{\abs{\Pi_{p\to q}(s(p))-s(q)}}{\dist(p,q)^\alpha},
\end{equation}
where $\Pi_{p\to q}\colon E_p\to E_q$ is the parallel transport along the minimizing geodesic from $p$ to $q$,
and we set
\begin{equation*}
	\norm{s}_{C^{k,\alpha}}
	:=\sum_{m=0}^{k-1}\sup\abs{\nabla^ms}+\norm{\nabla^ks}_{C^{\alpha}}.
\end{equation*}
For $s\in C^{k,\alpha}_\delta(E)$, we define
\begin{equation*}
	\norm{s}_{C^{k,\alpha}_\delta}:=\norm{x^{-\delta}s}_{C^{k,\alpha}}.
\end{equation*}
Note that the norm \eqref{eq:Holder-norm} remains equivalent when
the supremum over $\dist(p,q)<r_\mathrm{inj}$ is replaced with
the supremum over $\dist(p,q)<r$ for any fixed parameter $r\in(0,r_\mathrm{inj})$.
An alternative way to define the H\"older spaces above is to use
a system of uniformly locally finite coordinate charts $\set{(U_\lambda,\Phi_\lambda)}$,
where $\Phi_\lambda$ maps $U_\lambda\subset X$ onto the open ball $B_r\subset\mathbb{R}^{2n}$
of radius $r>0$ (independent of $\lambda$), such that
\begin{equation}
	\label{eq:bounded-geometry-coordinates}
	c^{-1}\tensor{\delta}{_i_j}<\tensor{(g_0)}{_i_j}<c\tensor{\delta}{_i_j},\qquad
	\abs{\partial_{k_1}\dots\partial_{k_m}\tensor{(g_0)}{_i_j}}<c_m
\end{equation}
holds in each chart and $X$ is covered by $\set{\Phi_\lambda^{-1}(B_{r/2}(0))}$.
In particular, one can use a version of Lee's ``M\"obius coordinates'' \cite{Lee-06} adapted to the ACH case
to describe the H\"older spaces (this approach is taken in \cite{Roth-99-Thesis}).
A consequence of this is that the H\"older spaces for $g_0$ are irrelevant to the choice of $(J,\theta)$.

The Cheng--Yau metric $g$ on a smoothly bounded strictly pseudoconvex domain $\Omega$
can be regarded as an ACH metric in the sense of Definition \ref{dfn:ACH-metrics}
via the ``square root construction'' of Epstein--Melrose--Mendoza~\cite{Epstein-Melrose-Mendoza-91}
(see also~\cite{Matsumoto-16}*{Example 2.4}).
In this case, the compactification $\overline{X}$ is topologically $\overline{\Omega}$,
but the $C^\infty$-structure is replaced so that the square roots of boundary defining functions are
regarded as smooth. Therefore, in this case,
\begin{equation*}
	C^{k,\alpha}_\delta(S^2T^*X)=\varphi^{\delta/2}C^{k,\alpha}(S^2T^*X),
\end{equation*}
where $\varphi$ is any smooth positive defining function of $\Omega$.
Although the identity map $\iota\colon\overline{X}\to\overline{\Omega}$ does not have smooth inverse,
its restriction to the interior $X$ and to the boundary $\bdry\overline{X}$ are diffeomorphisms onto $\Omega$ and
$\bdry\Omega$, respectively. Thus $\bdry\overline{X}$ is identified with $\bdry\Omega$.
The conformal infinity of $g$ is actually the natural CR structure on $\bdry\Omega$, i.e.,
the one induced from the complex structure of the ambient Stein manifold.

\subsection{Einstein deformations}
\label{subsec:Einstein-deformations}

Let $g$ be an arbitrary ACH metric on $X$ satisfying the Einstein equation, which is in fact forced to be
\begin{equation*}
	\Ric(g)=-(n+1)g.
\end{equation*}
The conformal infinity of $g$, which is a partially integrable CR structure on $M=\bdry\overline{X}$,
is denoted by $(H,J)$.

Take a neighborhood of $J$ in the set of all $C^{2,\alpha}$ partially integrable CR structures
admitted by the contact distribution $H$.
This is identified with a neighborhood of the origin in the H\"older space $C^{2,\alpha}(S^2(\wedge^{1,0}M))$
of anti-hermitian symmetric 2-tensors over $T^{1,0}M$.
The identification is given as follows:
For $\psi\in C^{2,\alpha}(S^2(\wedge^{1,0}M))$ with sufficiently small norm
(with respect to the Levi form $L_\theta$ for $J$ and some fixed pseudohermitian structure $\theta$),
expressed as $\tensor{\psi}{_\alpha_\beta}$ with respect to a local frame $\set{Z_\alpha}$ of $T^{1,0}M$,
the vector bundle
\begin{equation*}
	\fourIdx{\psi}{}{}{}{\smash{T^{1,0}}}M
	=\Span\set{Z_\alpha+\tensor{\psi}{_\alpha^{\conj{\beta}}}Z_{\conj{\beta}}}
\end{equation*}
defines the corresponding $C^{2,\alpha}$ partially integrable CR structure $J_\psi$,
where the second index of $\psi$ is raised by the Levi form and Einstein's summation convention is observed.
We write $J_\psi=J+\psi$.

We first consider a preliminary extension of $J_\psi$ to an ACH metric $g_\psi$.
Identify a neighborhood of $\bdry\overline{X}$ with $M\times[0,\varepsilon)$ so that $g$ is as in
\eqref{eq:ACH-definition}.
Fix a cutoff function $\chi\in C^\infty(\overline{X})$ that equals $1$ near $\bdry\overline{X}$
and is supported in $M\times[0,\varepsilon)$.
Let $\tilde{\psi}$ be the difference of the Levi forms of $J$ and $J_\psi$ with respect to $\theta$, and set
\begin{equation*}
	g_\psi=g+\frac{\chi\tilde{\psi}}{x^2}.
\end{equation*}
If $\norm{\psi}_{C^{2,\alpha}}$ is sufficiently small, then $g_\psi$ becomes an (in general non-smooth) ACH metric
with conformal infinity $J_\psi$. Note that $g_0=g$.
One can verify that the metric $g_\psi$ is approximately Einstein in the sense that
\begin{equation*}
	\Ric(g_\psi)+(n+1)g_\psi\in C_1^{0,\alpha}(S^2T^*X);
\end{equation*}
hence, if $\mathcal{E}$ is defined by \eqref{eq:gauged-Einstein-equation} with $\lambda=-(n+1)$,
then it is trivial that $\mathcal{E}_{g_\psi}(0)\in C_1^{0,\alpha}(S^2T^*X)$.

We define a map
\begin{equation}
	Q\colon\mathcal{B}\to C^{2,\alpha}(S^2(\wedge^{1,0}M))\oplus C_\delta^{0,\alpha}(S^2T^*X),
\end{equation}
where $0<\delta\le 1$ and $\mathcal{B}$ is a small neighborhood of
$0\in C^{2,\alpha}(S^2(\wedge^{1,0}M))\oplus C_\delta^{2,\alpha}(S^2T^*X)$, by
\begin{equation}
	Q(\psi,h)=(\psi,\mathcal{E}_{g_\psi}(h)).
\end{equation}
We shall prove that this is bijective near the origin; then the inverse image of $(\psi,0)$ contains
only one element $(\psi,h_\psi)$ close to the origin,
for which the metric $g_\psi+h_\psi$ is an Einstein ACH metric with conformal infinity $J_\psi$
by~\cite{Biquard-00}*{Lemme I.1.4}.
By the inverse function theorem,
it suffices to show that the linearization of $Q$ is an isomorphism between the Banach spaces.
Since the first component of $Q$ is just the identity map, what we have to verify is that
\begin{equation}
	\label{eq:linearized-Einstein-on-decay-1}
	\mathcal{E}_g'\colon C_\delta^{2,\alpha}(S^2T^*X)\to C_\delta^{0,\alpha}(S^2T^*X)
\end{equation}
is isomorphic.

The Fredholm theory of ACH metrics
shows that the operator \eqref{eq:linearized-Einstein-on-decay-1} is an isomorphism if and only if
the $L^2$ kernel vanishes.
Here the $L^2$ kernel, denoted by $\ker_{(2)}\mathcal{E}_g'$,
is the kernel of $\mathcal{E}_g'$ understood as an unbounded operator with domain
\begin{equation*}
	\dom\mathcal{E}_g':=\set{\alpha\in L^2(S^2T^*X)|\mathcal{E}_g'\alpha\in L^2(S^2T^*X)},
\end{equation*}
which is called the maximal closed extension.
Thus the space $\ker_{(2)}\mathcal{E}_g'$ is called the \emph{obstruction space} of Einstein deformations,
and when it vanishes $g$ is called \emph{nondegenerate}.
We can summarize the discussion so far as in the next proposition
(the uniqueness statement follows from~\cite{Biquard-00}*{Proposition I.4.6}).

\begin{prop}[Roth \cite{Roth-99-Thesis}*{Theorem 1.1}, Biquard \cite{Biquard-00}*{Th\'eor\`eme I.4.8}]
	\label{prop:possibility-of-deformation}
	Let $g$ be a nondegenerate Einstein ACH metric with conformal infinity $J$.
	Then, for any $\psi\in C^{2,\alpha}(S^2(\wedge^{1,0}M))$ close enough to zero,
	there exists an Einstein ACH metric $g'$ whose conformal infinity is $J_\psi=J+\psi$.
	The metric $g'$ is locally unique in the sense that
	any Einstein ACH metric lying in a sufficiently small $C^{2,\alpha}_\delta$-neighborhood
	of $g'$ pulls back to $g'$ by a diffeomorphism on $X$ inducing the identity on $\bdry\overline{X}$.
\end{prop}

In the rest of this subsection, we describe the reason why $\ker_{(2)}\mathcal{E}_g'=0$ implies the
isomorphicity of \eqref{eq:linearized-Einstein-on-decay-1}
by putting it in the context of
general theory of geometrically defined elliptic linear differential operators for ACH metrics.

Let $P\colon\Gamma(E)\to\Gamma(F)$ be an elliptic differential operator of order $m$ on a manifold equipped
with an ACH metric,
where $E$ and $F$ are subbundles of $(TX)^{\otimes s}\otimes(T^*X)^{\otimes t}$
invariant under the action of $O(2n)$
(the group $O(2n)$ can be replaced by $U(n)$ for K\"ahler ACH metrics),
with a universal expression in terms of the Levi-Civita covariant
differentiation and the actions of the curvature tensor of $g$.
Such an operator is called \emph{geometric}.
A consequence of this assumption is that $P$ determines a well-defined mapping
\begin{equation*}
	C_\delta^{m,\alpha}(E)\to C_\delta^{0,\alpha}(E),
\end{equation*}
and also
\begin{equation*}
	H^m_\delta(E)\to L^2_\delta(E),
\end{equation*}
for an arbitrary $\delta\in\mathbb{R}$. Here $L^2_\delta(E)$ and $H^m_\delta(E)$ are weighted $L^2$ and
$L^2$-Sobolev spaces,
which are defined by $L^2_\delta(E):=x^\delta L^2(E)$ and $H^k_\delta(E):=x^\delta H^k(E)$.

Another virtue of the geometricity is that it allows us to consider the operator $P$
on any ACH manifold simultaneously.
In particular, $P$ makes sense on $\mathbb{C}H^n$.
Then we can formulate the following coerciveness assumption, which is crucial in the next lemma:
\begin{equation}
	\label{eq:coercivity}
	\norm{\alpha}^2\le C\norm{P\alpha}^2,\qquad\alpha\in\dom P\subset L^2(E),\qquad\text{on $\mathbb{C}H^n$}.
\end{equation}
Its validity for $P=\mathcal{E}_g'$ is obvious from \eqref{eq:linearized-Einstein-formula}.

We need in addition to introduce a nonnegative real number $R_P$ called the \emph{indicial radius} of $P$.
For this, we again consider the operator $P$ on $\mathbb{C}H^n$.
Fixing an origin $o\in\mathbb{C}H^n$ identifies a subgroup $G=U(n)$ of the group of isomorphisms $G_0=\PSU(n,1)$,
which gives the expression
\begin{equation*}
	\mathbb{C}H^n\cong G_0/G.
\end{equation*}
Then $E$ and $F$ are expressed as the associated bundles $G_0\times_GV$ and $G_0\times_GW$, respectively,
where $V$ and $W$ are representations of $G$.
Furthermore, we fix a unit tangent vector $v_0\in T_o\mathbb{C}H^n$ at the origin $o$,
and let $\gamma(r)=\exp(rv_0)$ be the geodesic that it determines.
Since $G$ acts on $T_o\mathbb{C}H^n$, $v_0$ specifies the isotropy subgroup $H=U(n-1)\subset G$.
To state it differently, $H$ is the isotropy subgroup of the $G$-action on the sphere at infinity $S^{2n-1}$
about the limit point of $\gamma$. Thus
\begin{equation*}
	S^{2n-1}\cong G/H.
\end{equation*}
We introduce an identification
\begin{equation*}
	\mathbb{C}H^n\setminus\set{o}\cong S^{2n-1}\times (0,\infty)
\end{equation*}
that maps $g\cdot\gamma(r)$ to the pair $(gH,r)$, where $g\in G$.
Then, restricted on each concentric sphere $S_r:=S^{2n-1}\times\set{r}$,
the vector bundle $E$ can be seen as a homogeneous bundle over $G/H$.
If we identify the fiber $E_{\gamma(r)}$ with $V=E_o$ by the parallel transport along $\gamma$,
then
\begin{equation*}
	E|_{S_r}\cong G\times_HV.
\end{equation*}
Therefore, a section of $E$ (over $\mathbb{C}H^n\setminus\set{o}$) can be identified with
a function on $G\times(0,\infty)$ with values in $V$ that is $H$-equivariant.
A similar identification can be introduced for sections of $F$.

We can write $P$ down in terms of this expression of sections of $E$ and $F$
as vector-valued functions on $G\times (0,\infty)$.
Then it follows from \cite{Biquard-00}*{Equation (I.1.2)}
that the derivatives tangent to $G$ vanish at the limit $r\to\infty$.
In fact, $P$ has the following form, where each $a_i(r)$ is a function with values in the space $\Hom_H(V,W)$
of $H$-equivariant linear mappings $V\to W$:
\begin{equation*}
	P=a_0(r)\partial_r^m+a_1(r)\partial_r^{m-1}+\dots+a_m(r)+O(e^{-r}).
\end{equation*}
As $r\to\infty$, the coefficients $a_i(r)$ have well-defined limits $a_i\in\Hom_H(V,W)$.
A number $s\in\mathbb{C}$ is called an \emph{indicial root} of $P$ when
\begin{equation*}
	I_P(s):=(-1)^ma_0s^m+(-1)^{m-1}a_1s^{m-1}+\dots+a_m\in\Hom_H(V,W)
\end{equation*}
fails to be injective.
Intuitively, the set $\Sigma_P$ of indicial roots is such that any solution of $Pu=0$ is expected to behave
asymptotically like $u\sim u_0e^{-sr}$ for some $s\in\Sigma_P$ and a section $u_0$ over $S^{2n-1}$.

For the formal adjoint operator $P^*\colon\Gamma(F)\to\Gamma(E)$, we can show that
\begin{equation*}
	I_{P^*}(s)=I_P(2n-\conj{s})^*.
\end{equation*}
The number $2n$ here can be understood as the twice the borderline weight of being in the $L^2$ space.
(See also Lee \cite{Lee-06}*{Proposition 4.4} for the AH case,
in which the ``weight of the tensor bundle'' appears in the formula because of a slight difference
in the definition of indicial roots.)
In particular, when $P$ is formally self-adjoint, then the indicial roots appear symmetrically
about the line $\re s=n$.
In this case,
\begin{equation*}
	R_P:=\min_{s\in\Sigma_P}\abs{\re s-n}
\end{equation*}
is called the \emph{indicial radius} of $P$.
Now we can formulate the following proposition.

\begin{prop}[Biquard~\cite{Biquard-00}*{Proposition I.3.5}]
	\label{prop:apriori}
	Assume that the operator $P$ acting on sections of $E$ is formally self-adjoint and
	satisfies the coerciveness estimate \eqref{eq:coercivity} on $\mathbb{C}H^n$.
	Then, for $\abs{\delta}<R_P$, the operator $P$ seen as mappings
	\begin{equation}
		\label{eq:Holder-decay}
		C_{n+\delta}^{m,\alpha}(E)\to C_{n+\delta}^{0,\alpha}(E)
	\end{equation}
	and
	\begin{equation}
		\label{eq:Sobolev-decay}
		H^m_\delta(E)\to L^2_\delta(E)
	\end{equation}
	are Fredholm with index zero, and the kernel of each of the mappings above equals $\ker_{(2)}P$.
\end{prop}

Recall that the estimate~\eqref{eq:coercivity} for the linearized gauged Einstein operator $\mathcal{E}_g'$
follows from \eqref{eq:linearized-Einstein-formula}.
A computation of the indicial roots of $\mathcal{E}'_g$ is given (or at least sketched)
in~\cite{Biquard-00}*{Section I.2.A and Lemme I.4.3}, whose result is as follows.
The $G$-representation associated to $E=S^2T^*X$ is
\begin{equation*}
	V=S_\mathbb{R}^2\mathfrak{m}_0^*,
\end{equation*}
where $\mathfrak{g}_0=\mathfrak{g}\oplus\mathfrak{m}_0$ is the Cartan decomposition of
$\mathfrak{g}_0=\mathfrak{su}(n,1)$ for the symmetric space $\mathbb{C}H^n=G_0/G$.
The space $\mathfrak{m}_0$ is canonically identified with $T_o\mathbb{C}H^n$, and
decomposes as $\mathbb{C}v_0\oplus\mathbb{C}^{n-1}$ by the $H$-action.
Since $a_i$'s are $H$-equivariant, by Schur's Lemma each $H$-irreducible component of
$S_\mathbb{R}^2\mathfrak{m}_0^*$ is mapped by the indicial polynomial $I_{\mathcal{E}'_g}(s)$ into
the sum of isomorphic components.
On $S_\mathbb{C}^2(\mathbb{C}^{n-1})^*\to S_\mathbb{C}^2(\mathbb{C}^{n-1})^*$,
the indicial polynomial becomes $s^2-2ns$ times a nonzero constant, which gives indicial roots $0$ and $2n$.
The fact is that these are the closest roots to the borderline $\re s=n$;
hence we conclude that $R_{\mathcal{E}'_g}=n$.

While the computation in~\cite{Biquard-00} is described in terms of relevant Lie algebras,
a more primitive, though probably less insightful, calculation is given
in~\cite{Matsumoto-16}*{Lemma 5.4} (the operator $\Delta_\mathrm{L}+n+2$ that appears
in~\cite{Matsumoto-16} is
nothing but $\mathcal{E}_g'$ up to a constant factor; note that ACH manifolds in~\cite{Matsumoto-16}
has real dimension $2n+2$).
Let us recall this computation.
Rather than using the unit ball model, here we identify $\mathbb{C}H^n$ with $\mathcal{H}^{2n-1}\times(0,\infty)$,
where $\mathcal{H}^{2n-1}$ is the Heisenberg group, so that the complex hyperbolic metric is given by
formula~\eqref{eq:ACH-model} with the standard pseudohermitian structure $\theta$.
Let $T$ be the Reeb vector field on $\mathcal{H}^{2n-1}$ and $\set{Z_1,\dots,Z_{n-1}}$ a local frame of
the CR holomorphic tangent bundle. If we set
$\bm{Z}_\tau:=\frac{1}{2}x\partial_x+ix^2T$ and $\bm{Z}_\alpha:=xZ_\alpha$, $\alpha=1$, $\dots$, $n-1$,
then $\set{\bm{Z}_\tau,\bm{Z}_1,\dots,\bm{Z}_{n-1}}$ is a local frame of $T^{1,0}\mathbb{C}H^n$ and
the Christoffel symbols are given by~\cite{Matsumoto-16}*{Equations (5.2)}.
We compute the action of $\mathcal{E}_g'$ using this frame.
Then we obtain, for example,
\begin{equation*}
	\tensor{(\mathcal{E}_g'\sigma)}{_\alpha_\beta}
	=-\frac{1}{4}x\partial_x(x\partial_x-2n)\tensor{\sigma}{_\alpha_\beta}
	+O(x)\cdot(\text{derivatives of $\sigma$ in the $\mathcal{H}^{2n-1}$-direction}).
\end{equation*}
One can conclude from this that the indicial roots appearing from the component
$S_\mathbb{C}^2(\mathbb{C}^{n-1})^*\to S_\mathbb{C}^2(\mathbb{C}^{n-1})^*$ are $0$ and $2n$,
and the other roots can also be read off from \cite{Matsumoto-16}*{Equations (5.9)} similarly.

Anyway the indicial radius of $\mathcal{E}'_g$ is $n$ and, by Proposition \ref{prop:apriori},
the mappings \eqref{eq:Holder-decay} and \eqref{eq:Sobolev-decay} for the linearized gauged Einstein
operator are isomorphic for $\abs{\delta}<n$ if the $L^2$ kernel vanishes.
Proposition \ref{prop:possibility-of-deformation} follows by taking $\delta$ close to $-n$.
We also remark that the conclusion of Proposition \ref{prop:apriori} for $\delta>0$ can be seen as giving
an improved decay of the elements of $\ker_{(2)}\mathcal{E}_g'$, which has applications in the proof of
the main theorem for $n=3$ and in the appendix.

\section{Infinitesimal Einstein deformations and $L^2$ cohomology}
\label{sec:infinitesimal-Einstein-deformaions}

\subsection{Infinitesimal Einstein deformations on K\"ahler manifolds}

Let us consider the linearized gauged Einstein operator $\mathcal{E}_g'$ of
a complete Einstein K\"ahler metric $g$ with Einstein constant $\lambda<0$ defined on a complex manifold
$\Omega$ with dimension $n$.
Thanks to the complex structure, any symmetric 2-tensor $\sigma\in\Gamma(S^2T^*\Omega)$ decomposes
into the sum of the Hermitian and the anti-Hermitian parts: $\sigma=\sigma_H+\sigma_A$.
By definition, the two summands satisfy
\begin{equation*}
	\sigma_H(J\cdot,J\cdot)=\sigma_H(\cdot,\cdot)\quad\text{and}\quad
	\sigma_A(J\cdot,J\cdot)=-\sigma_A(\cdot,\cdot),
\end{equation*}
where $J$ denotes the almost complex structure endomorphism.
The Hermitian (resp.~anti-Hermitian) part of $S^2T^*\Omega$ will be denoted by
$S_H^2T^*\Omega$ (resp.~$S_A^2T^*\Omega$).
This decomposition is respected by $\mathcal{E}_g'$,
for the curvature of the K\"ahler metric has only components of
the type $\tensor{R}{_i_{\conj{j}}_k_{\conj{l}}}$.

We discuss the action of $\mathcal{E}_g'$ on each component based on
Koiso's observation~\cite{Koiso-83}*{Section 7} (see also Besse~\cite{Besse-87}*{Section 12.J}).
First, we identify any Hermitian symmetric form $\sigma_H$ with the differential $(1,1)$-form
$\sigma_H(\cdot,J\cdot)$, which is denoted by $\sigma_H\circ J$.
Then the action of $\mathcal{E}_g'$ on the Hermitian part is related to that of the Hodge--de Rham Laplacian
$\Delta_d$ as follows:
\begin{equation*}
	(\mathcal{E}_g'\sigma_H)\circ J=\frac{1}{2}(\Delta_d-2\lambda)(\sigma_H\circ J).
\end{equation*}
On complete manifolds, by a result of Gaffney~\cite{Gaffney-55} it is known that
$\Delta_d$ is essentially self-adjoint, meaning that it has unique self-adjoint extension.
In particular, the maximal closed extension of $\Delta_d$ agrees with $dd^*+d^*d$ (where $d$ also acts
distributionally).
Therefore, since $\lambda<0$, it follows that
\begin{equation}
	\label{eq:linearized-Einstein-hermitian}
	\ker_{(2)}\mathcal{E}_g'\cap L^2(S^2_HT^*\Omega)=0.
\end{equation}
Second, the action of $\mathcal{E}_g'$ on the anti-Hermitian part $\sigma_A$ is reinterpreted as follows.
Let $\sigma_A=\sigma_A^{2,0}+\sigma_A^{0,2}$ be the type decomposition of $\sigma_A$,
and we identify $\sigma_A^{0,2}$ through the metric duality with a
$(0,1)$-form with values in $T^{1,0}=T^{1,0}\Omega$, which is denoted by $g^{-1}\circ\sigma_A^{0,2}$. Then
\begin{equation}
	\label{eq:linearized-Einstein-anti-hermitian}
	g^{-1}\circ(\mathcal{E}_g'\sigma_A)^{0,2}=\frac{1}{2}\Delta_\partialbar(g^{-1}\circ\sigma_A^{0,2}).
\end{equation}
By \eqref{eq:linearized-Einstein-hermitian} and \eqref{eq:linearized-Einstein-anti-hermitian},
we have a natural identification
\begin{equation}
	\label{eq:reduction-to-harmonics}
	\ker_{(2)}\mathcal{E}_g'\cong \mathcal{H}_{(2)}^{0,1}(T^{1,0}),
\end{equation}
where $\mathcal{H}_{(2)}^{0,1}(T^{1,0})$ is the space of $L^2$ harmonic $T^{1,0}$-valued $(0,1)$-forms:
\begin{equation*}
	\mathcal{H}_{(2)}^{0,1}(T^{1,0})
	:=\set{\alpha\in L^2\spacewedge^{0,1}(T^{1,0})|\text{$\partialbar\alpha=0$, $\partialbar^*\alpha=0$}}
	=\set{\alpha\in L^2\spacewedge^{0,1}(T^{1,0})|\Delta_\partialbar\alpha=0}.
\end{equation*}
The latter equality follows from the essential self-adjointness of $\Delta_\partialbar$
due to Chernoff~\cite{Chernoff-73}.

\subsection{Reduction to $L^2$ cohomology}
\label{subsec:reduction-to-L2-cohomology}

The Hodge--Kodaira decomposition on noncompact Hermitian manifolds reads as follows, where
$E$ is an arbitrary Hermitian holomorphic vector bundle:
\begin{equation*}
	L^2\spacewedge^{p,q}(\Omega;E)
	=\mathcal{H}_{(2)}^{p,q}(\Omega;E)\oplus\overline{\raisebox{1pt}{\mathstrut}\im\partialbar_{p,q-1}}
	\oplus\overline{\raisebox{1pt}{\mathstrut}\im\partialbar_{p,q}^*};
\end{equation*}
here
\begin{equation*}
	\partialbar=\partialbar_{p,q}\colon
	L^2\spacewedge^{p,q}(\Omega;E)\to L^2\spacewedge^{p,q+1}(\Omega;E)
\end{equation*}
is the maximal closed extension of $\partialbar$ acting on compactly supported smooth $(p,q)$-forms.
Therefore, the space $\mathcal{H}_{(2)}^{p,q}(\Omega;E)$ is isomorphic to the so-called \emph{reduced}
$L^2$ cohomology:
\begin{equation*}
	\mathcal{H}_{(2)}^{p,q}(\Omega;E)\cong
	H_{(2), \mathrm{red}}^{p,q}(\Omega;E)
	:=\ker\partialbar_{p,q}/\overline{\raisebox{1pt}{\mathstrut}\im\partialbar_{p,q-1}}.
\end{equation*}
The reduced cohomology can be different from the usual $L^2$ cohomology
\begin{equation*}
	H_{(2)}^{p,q}(\Omega;E):=\ker\partialbar_{p,q}/\im\partialbar_{p,q-1},
\end{equation*}
but it is clear that
$H_{(2)}^{p,q}(\Omega;E)=0$ implies $H_{(2), \mathrm{red}}^{p,q}(\Omega;E)=0$.
Therefore we can consider the usual $L^2$ cohomology to get a result on harmonic forms.

Let us recall an exact sequence for $L^2$ cohomologies.
Since an inclusion $K\subset K'$ between compact subsets of $\Omega$ induces a homomorphism
$H_{(2)}^{p,q}(\Omega\setminus K;E)\to H_{(2)}^{p,q}(\Omega\setminus K';E)$ by restriction,
we may define the inductive limit
\begin{equation*}
	\varinjlim_K H_{(2)}^{p,q}(\Omega\setminus K;E),
\end{equation*}
where $K$ runs through the compact subsets of $\Omega$.
Then we have the following exact sequence (cf.~Ohsawa~\cite{Ohsawa-91}):
\begin{equation*}
	\dotsb
	\to H_c^{p,q}(\Omega;E)\to H_{(2)}^{p,q}(\Omega;E)
	\to\varinjlim_K H_{(2)}^{p,q}(\Omega\setminus K;E)
	\to H_c^{p,q+1}(\Omega;E)\to\dotsb.
\end{equation*}
Here $H_c^{p,q}(\Omega;E)$ denotes the cohomology with compact support.

Now suppose that $\Omega$ is a Stein manifold.
Then $H^{0,1}_c(\Omega;E)$ vanishes for any holomorphic vector bundle $E$,
because it is the dual vector space of $H^{n,n-1}(\Omega;E^*)$, which vanishes by the Oka--Cartan theorem
(alternatively, Andreotti--Vesentini \cite{Andreotti-Vesentini-65}*{Theorem 5} gives
a direct differential-geometric proof).
Therefore, by \eqref{eq:reduction-to-harmonics} and the exact sequence above,
the vanishing of $\ker_{(2)}\mathcal{E}_g'$ follows once
\begin{equation}
	\label{eq:reduction-to-boundary-cohomology}
	\varinjlim_K H_{(2)}^{0,1}(\Omega\setminus K;T^{1,0})=0
\end{equation}
is shown.

Let us further suppose that $g$ is an ACH K\"ahler-Einstein metric.
In this case, we apply Proposition \ref{prop:apriori} to show that
$\ker_{(2)}\mathcal{E}_g'$ actually lies in the weighted $L^2$-space
$L^2_\delta(S^2T^*\Omega)$ for $0<\delta<n$, which implies that
\begin{equation}
	\label{eq:apriori-decay-for-complex-deformation}
	\mathcal{H}_{(2)}^{0,1}(\Omega;T^{1,0})\subset L^2_\delta\spacewedge^{0,1}(\Omega;T^{1,0}).
\end{equation}
Thus we are led to considering the weighted $L^2$ cohomology.
The vanishing of the weighted cohomology $H_{(2), \delta}^{0,1}(\Omega;T^{1,0})$ follows from
\begin{equation*}
	\varinjlim_K H_{(2), \delta}^{0,1}(\Omega\setminus K;T^{1,0})=0
\end{equation*}
because the weighted cohomology $H_{(2), \delta}^{p,q}(\Omega;T^{1,0})$ is nothing but
$H_{(2)}^{p,q}(\Omega;E_\delta)$, where $E_\delta$ denotes the vector bundle $T^{1,0}$ equipped with the metric
$\varphi^{-\delta}g$.
Now suppose that $H_{(2), \delta}^{0,1}(\Omega;T^{1,0})=0$ is known, and
take any element $\alpha\in \mathcal{H}_{(2)}^{0,1}(\Omega;T^{1,0})$.
From \eqref{eq:apriori-decay-for-complex-deformation} it follows that
$\alpha\in L^2_\delta\spacewedge^{0,1}(\Omega;T^{1,0})$,
and at the same time we have $\partialbar\alpha=0$.
Hence, by the assumption,
there is some $\beta\in L^2_\delta(\Omega;T^{1,0})$ for which $\alpha=\partialbar\beta$.
Then it turns out that $\beta$ also belongs to $L^2(\Omega;T^{1,0})$.
Now since $\partialbar^*\alpha=0$, we obtain $\partialbar^*\partialbar\beta=0$,
which implies $(\partialbar\beta,\partialbar\beta)=0$ and hence $\alpha=\partialbar\beta=0$.
Thus we can conclude that $\mathcal{H}_{(2)}^{0,1}(\Omega;T^{1,0})$ vanishes, and so does the obstruction
space $\ker_{(2)}\mathcal{E}_g'$.

\section{Proof of main theorem}
\label{sec:proof-main-theorem}

We shall prove that
\begin{equation}
	\label{eq:reduction-to-boundary-weighted-cohomology}
	\varinjlim_K H_{(2), \delta}^{0,1}(\Omega\setminus K;T^{1,0})=0
\end{equation}
holds for $n\ge 3$;
thus our main theorem follows by the discussion in the previous section and
Proposition \ref{prop:possibility-of-deformation}.
In the course of the proof of \eqref{eq:reduction-to-boundary-weighted-cohomology},
we will also see that \eqref{eq:reduction-to-boundary-cohomology} holds when $n\ge 4$.
Therefore, the only case one really has to consider the weighted cohomology is when $n=3$.

Since the $L^2$ cohomology is invariant for quasi-equivalent metrics,
we can replace the Cheng--Yau metric $g$ with the metric $\tilde{g}$ expressed as
$\tensor{\tilde{g}}{_i_{\conj{j}}}=\partial_i\partial_{\conj{j}}(-\log\tilde{\varphi})$,
where $\tilde{\varphi}\in C^\infty(\overline{\Omega})$ is a \emph{smooth} positive defining function.
This simplification avoids annoying differentiability issues.
In what follows, we omit tildes: $\tilde{g}$ and $\tilde{\varphi}$ are simply denoted by
$g$ and $\varphi$, respectively.

\subsection{Preliminary considerations}

We define $\mathcal{U}_\rho:=\set{0<\varphi<\rho}\subset\Omega$ for small $\rho>0$
so that $M_\rho=\set{\varphi=\rho}$ is smooth.
What we prove in this section is actually the following,
which is supposedly stronger than \eqref{eq:reduction-to-boundary-weighted-cohomology}.

\begin{prop}
	\label{prop:boundary-vanishing}
	Let $n\ge 3$. For any positive number $\delta>0$, if $\rho>0$ is sufficiently small then
	\begin{equation}
		\label{eq:boundary-weighted-vanishing}
		H_{(2), \delta}^{0,1}(\mathcal{U}_\rho;T^{1,0})=0
	\end{equation}
	holds.
\end{prop}

This claims the solvability of a $\partialbar$-equation
on a complete manifold under the presence of boundary. The proof reduces to establishing the estimate below
(see H\"ormander~\cite{Hormander-65}*{Theorem 1.1.4} or~\cite{Hormander-90}*{Theorem 4.1.1}).

\begin{prop}
	\label{prop:boundary-weighted-estimate}
	Let $n\ge 3$ and $\delta>0$. For sufficiently small $\rho>0$, there exists a constant $C>0$ such that
	\begin{equation}
		\label{eq:boundary-weighted-estimate}
		\norm{\alpha}^2\le C(\norm{\partialbar\alpha}^2+\norm{\partialbar^*\alpha}^2),
		\qquad \alpha\in\dom\partialbar\cap\dom\partialbar^*
		\subset L^2_\delta\spacewedge^{0,1}(\mathcal{U}_\rho;T^{1,0}).
	\end{equation}
\end{prop}

As remarked in the previous section, we may incorporate the weight
into the fiber metric of $T^{1,0}$. Therefore we shall present the necessary computation
for differential forms on $\Omega$ with values in an arbitrary Hermitian holomorphic vector bundle $E$.
In addition, we consider differential $(0,q)$-forms in general, $0\le q\le n$, to make the situation
clearer.

While the domain of $\partialbar=\partialbar_{0,q}$ contains the space
$C^\infty_c\spacewedge^{0,q}(\overline{\mathcal{U}}_\rho;E)$ of $E$-valued smooth $(0,q)$-forms
with compact support in $\overline{\mathcal{U}}_\rho=\set{0<\varphi\le\rho}$ as a subspace,
the domain of $\partialbar^*=\partialbar_{0,q-1}^*$ does not
(unless $q=0$, for which $\partialbar^*$ is trivial). We define
\begin{equation*}
	\mathcal{D}^{0,q}(\overline{\mathcal{U}}_\rho;E)
	:=C^\infty_c\spacewedge^{0,q}(\overline{\mathcal{U}}_\delta;E)\cap\dom\partialbar^*.
\end{equation*}
This space is described as follows.
Let $\xi$ be the $(1,0)$-vector field on $\overline{\mathcal{U}}_\rho$ such that,
for each $0<c\le\rho$, its restriction $\xi|_{M_c}$ along the level set $M_c=\set{\varphi=c}$
is the unit normal vector field pointing toward $\bdry\Omega$.
Then $\alpha\in C^\infty_c\spacewedge^{0,q}(\overline{\mathcal{U}}_\rho;E)$ belongs to
$\mathcal{D}^{0,q}(\overline{\mathcal{U}}_\rho;E)$ if and only if
\begin{equation}
	\label{eq:Neumann-condition}
	\iota_{\conj{\xi}}\alpha=0\quad\text{on $M_\rho$}.
\end{equation}
The following lemma shows that it suffices to establish the estimates for elements of
$\mathcal{D}^{0,q}(\overline{\mathcal{U}}_\rho;E)$.

\begin{lem}
	\label{lem:density}
	The space $\mathcal{D}^{0,q}(\overline{\mathcal{U}}_\rho;E)$ is dense in
	$\dom\partialbar\cap\dom\partialbar^*\subset L^2\spacewedge^{0,q}(\mathcal{U}_\rho;E)$
	with respect to the graph norm
	$\alpha\mapsto(\norm{\alpha}^2+\norm{\partialbar\alpha}^2+\norm{\partialbar^*\alpha}^2)^{1/2}$.
\end{lem}

\begin{proof}
	By a partition of unity, we may decompose $\alpha\in\dom\partialbar\cap\dom\partialbar^*$
	into the sum $\alpha=\alpha_1+\alpha_2$, where $\alpha_1$ is supported near $M_\rho$ and
	$\supp\alpha_2\subset\mathcal{U}_\rho$.
	It suffices to approximate $\alpha_1$ and $\alpha_2$ separately by elements of
	$\mathcal{D}^{0,q}(\overline{\mathcal{U}}_\rho;E)$.
	Further partition allows us to assume that $\alpha_1$ is supported in a local boundary chart $U$ of
	$\overline{\mathcal{U}}_\rho$.
	Then a result of H\"ormander~\cite{Hormander-65}*{Proposition 1.2.4}
	(see also Chen--Shaw~\cite{Chen-Shaw-01}*{Lemma 4.3.2}) shows that
	there exists a sequence $\alpha_1^\nu\in\mathcal{D}^{0,q}(\overline{\mathcal{U}}_\rho;E)$ supported in $U$
	such that $\alpha_1^\nu\to\alpha_1$ in the graph norm.
	The second term $\alpha_2$ is approximated by smooth forms supported in $\mathcal{U}_\rho$
	by the standard cut-off technique for complete manifolds
	(see, e.g., the proof of~\cite{Andreotti-Vesentini-65}*{Lemma 4}).
\end{proof}

We will later need the divergence of $\xi$ and the commutator $[\xi,\conj{\xi}]$,
which can be computed as follows.
Recall from Lee--Melrose~\cite{Lee-Melrose-82}*{Section 2}
that there exists a unique $(1,0)$-vector field $X$ on a (two-sided) neighborhood of $\bdry\Omega$ satisfying
\begin{equation*}
	\iota_X\partial\partialbar\varphi=\kappa\partialbar\varphi,\qquad
	\partial \varphi(X)=-1
\end{equation*}
for some real-valued function $\kappa$, which is called the \emph{transverse curvature}.
Then, since
\begin{equation*}
	g=\frac{\partial\varphi\,\conj{\partial}\varphi}{\varphi^2}-\frac{\partial\conj{\partial}\varphi}{\varphi},
\end{equation*}
we get $\abs{X}^2=\varphi^{-2}(1+\kappa\varphi)$ and hence $\xi=(1+\kappa\varphi)^{-1/2}\varphi X$,
which is the metric dual of $(1+\kappa\varphi)^{1/2}\partialbar(-\log\varphi)$.
This implies that
\begin{equation}
	\label{eq:divergence-of-unit-normal}
	\begin{split}
		\div\xi=\tr\nabla'\xi
		&=\tr_g\partial((1+\kappa\varphi)^{1/2}\partialbar(-\log\varphi))\\
		&=\tr_g\partial\partialbar(-\log\varphi)+o(1)
		=n+o(1)
	\quad\text{as $\varphi\to 0$},
	\end{split}
\end{equation}
where $\nabla=\nabla'+\nabla''$ is the type decomposition of the Levi-Civita connection. Moreover,
\begin{equation}
	\label{eq:bracket-of-unit-normal}
	[\xi,\conj{\xi}]
	=[\varphi X,\varphi\conj{X}]+o(1)
	=\varphi X-\varphi\overline{X}+\varphi^2[X,\overline{X}]+o(1)
	=\xi-\overline{\xi}+o(1)
	\quad\text{as $\varphi\to 0$},
\end{equation}
the last equality being because $[X,\overline{X}]$ is continuous up to $\bdry\Omega$ and hence has
$O(\varphi^{-1})$ pointwise norm with respect to $g$.

\subsection{The estimate}

The usual technique for obtaining estimates related to the $\partialbar$-Neumann problem
on strictly pseudoconvex domains is to use the Morrey--Kohn--H\"ormander equality,
which equates $\norm{\partialbar\alpha}^2+\norm{\partialbar^*\alpha}^2$ with
$\norm{\nabla''\alpha}^2$ plus zeroth-order terms and a boundary integral.
However, in our case, $M_\rho$ is strictly \emph{pseudoconcave} as the boundary of $\mathcal{U}_\rho$.
H\"ormander~\cite{Hormander-65} introduced (see also Folland--Kohn~\cite{Folland-Kohn-72}*{Section III.2})
``condition $Z(q)$'' to take such cases into consideration.
An interpretation of his technique is to use an equality that lies between those of Morrey--Kohn--H\"ormander
and Bochner--Kodaira--Nakano, the latter being, in this case, a relation between
$\norm{\partialbar\alpha}^2+\norm{\partialbar^*\alpha}^2$ and $\norm{\nabla'\alpha}^2$.
We shall apply his approach and write the relevant terms in terms of curvature.

We start with a geometric version of the Morrey--Kohn--H\"ormander equality
established by Andreotti--Vesentini~\cite{Andreotti-Vesentini-65}.
Let $\alpha$, $\beta\in\mathcal{D}^{0,q}(\overline{\mathcal{U}}_\rho;E)$.
Using local holomorphic coordinates $(z^1,\dots,z^n)$ and
a local holomorphic frame $(s_1,\dots,s_{\rank E})$ of $E$, we write
\begin{equation*}
	\alpha=\frac{1}{q!}\tensor{\alpha}{_{\conj{j}_1}_{\dotsb}_{\conj{j}_q}^a}
	d\conj{z}^{j_1}\wedge\dots\wedge d\conj{z}^{j_q}\otimes s_a,
\end{equation*}
where the sum is taken over all $(j_1,\dots,j_q)\in\set{1,\dots,n}^q$ (not only over the increasing indices)
and $a\in\set{1,\dots,\rank E}$, and $\tensor{\alpha}{_{\conj{j}_1}_{\dotsb}_{\conj{j}_q}^a}$ is skew-symmetric
in $j_1$, $\dotsc$, $j_q$.
Then we define
\begin{equation*}
	\braket{\alpha,\beta}:=\frac{1}{q!}
	\tensor{\alpha}{_{\conj{j}_1}_{\dotsb}_{\conj{j}_q}^a}
	\tensor{\conj{\beta}}{^{\conj{j}_1}^{\dotsb}^{\conj{j}_q}_a}
	\quad\text{and}\quad
	(\alpha,\beta):=\int_{\mathcal{U}_\rho}\braket{\alpha,\beta}dV_g.
\end{equation*}
(A more explicit notation for the latter may be $(\alpha,\beta)_{L^2(\mathcal{U}_\rho)}$,
but we suppress $L^2(\mathcal{U}_\rho)$ for notational simplicity.)
The $L^2$-norm of $\alpha$ on $\mathcal{U}_\rho$ is defined by $\norm{\alpha}=(\alpha,\alpha)^{1/2}$.
Moreover, we write
 $\abs{\alpha}^2=\braket{\alpha,\alpha}$ and
\begin{equation*}
	\norm{\alpha}_b^2:=\int_{M_\rho}\abs{\alpha}^2 dS_g,
\end{equation*}
where $dS_g$ is the area measure on $M_\rho$ induced by $dV_g$.
The actions of the Ricci tensor of $g$ and the curvature $S=\tensor{S}{_i_{\conj{j}}_a^b}$ of $E$
are defined as follows, where the square bracket notation means that we take the skew-symmetrization
over the indices $\conj{j}_1$, $\conj{j}_2$, ..., $\conj{j}_q$:
\begin{gather*}
	\tensor{(\Ric^\circ\alpha)}{_{\conj{j}_1}_{\dotsb}_{\conj{j}_q}^a}
	:=\sum_{s=1}^q\tensor{\Ric}{^{\conj{k}}_{\conj{j}_s}}
	\tensor{\alpha}{_{\conj{j}_1}_{\dotsb}_{\conj{k}}_{\dotsb}_{\conj{j}_q}^a}
	=q\tensor{\Ric}{^{\conj{k}}_{[\conj{j}_1|}}
	\tensor{\alpha}{_{\conj{k}|}_{\conj{j}_2}_{\dotsb}_{\conj{j}_q]}^a},\\
	\tensor{(\mathring{S}\alpha)}{_{\conj{j}_1}_{\dotsb}_{\conj{j}_q}^a}
	:=\sum_{s=1}^q\tensor{S}{^{\conj{k}}_{\conj{j}_s}_b^a}
	\tensor{\alpha}{_{\conj{j}_1}_{\dotsb}_{\conj{k}}_{\dotsb}_{\conj{j}_q}^b}
	=q\tensor{S}{^{\conj{k}}_{[\conj{j}_1|}_b^a}
	\tensor{\alpha}{_{\conj{k}|}_{\conj{j}_2}_{\dotsb}_{\conj{j}_q]}^b}.
\end{gather*}
Then, using \eqref{eq:Neumann-condition}, we get (see~\cite{Andreotti-Vesentini-65}*{p.~113})
\begin{equation*}
	\norm{\partialbar\alpha}^2+\norm{\partialbar^*\alpha}^2
	=\norm{\nabla''\alpha}^2+(\Ric^\circ\alpha,\alpha)+(\mathring{S}\alpha,\alpha)
	-q\int_{\bdry\mathcal{U}_\rho}\frac{1}{\abs{\partial\log\varphi}}\abs{\alpha}^2.
\end{equation*}
The asymptotic curvature behavior \eqref{eq:asymptotic-curvature-behavior} implies
\begin{equation*}
	\Ric^\circ\alpha=-q(n+1)\alpha+o(1).
\end{equation*}
Moreover, $\abs{\partial\log\varphi}=(1+\kappa\varphi)^{1/2}$ and thus it tends to $1$ uniformly at $\bdry\Omega$.
Therefore,
\begin{equation}
	\label{eq:Morrey-Kohn-Hormander}
	\norm{\partialbar\alpha}^2+\norm{\partialbar^*\alpha}^2
	=\norm{\nabla''\alpha}^2-q(n+1)\norm{\alpha}^2+(\mathring{S}\alpha,\alpha)
	-q\norm{\alpha}_b^2+o(\norm{\alpha}^2+\norm{\alpha}_b^2),
\end{equation}
where the remainder term being $o(\norm{\alpha}^2+\norm{\alpha}_b^2)$ means that,
for any $\varepsilon>0$, if $\rho>0$ is sufficiently small then
the absolute value of this term is bounded by $\varepsilon(\norm{\alpha}^2+\norm{\alpha}_b^2)$.

Andreotti--Vesentini equality \eqref{eq:Morrey-Kohn-Hormander} does not work well for our purpose,
for there is a negative boundary integral $-q\norm{\alpha}_b^2$ and the curvature term also becomes negative.
Both defects can be remedied by the following, which is H\"ormander's technique interpreted geometrically:
We decompose $\norm{\nabla''\alpha}^2$ into the tangential and normal parts,
\begin{equation*}
	\norm{\nabla''\alpha}^2=\norm{\nabla''_b\alpha}^2+\norm{\nabla_{\conj{\xi}}\alpha}^2,
\end{equation*}
and replace $\norm{\nabla''_b\alpha}^2$ with $\norm{\nabla'_b\alpha}^2$ by integration-by-parts.

\begin{lem}
	For $\alpha\in C^\infty_c\spacewedge^{0,q}(\overline{\mathcal{U}}_\rho;E)$,
	\label{lem:Hormander-technique}
	\begin{multline*}
		\norm{\nabla''_b\alpha}^2
		=\norm{\nabla'_b\alpha}^2+n(n+q-1)\norm{\alpha}^2-\norm{\iota_{\conj{\xi}}\alpha}^2
		-((\tr_gS)\alpha,\alpha)+(S(\xi,\conj{\xi})\alpha,\alpha)\\
		+2(n-1)\re(\nabla_{\conj{\xi}}\alpha,\alpha)+(n-1)\norm{\alpha}_b^2
		+o(\norm{\alpha}^2+\norm{\nabla_b\alpha}^2+\norm{\nabla_{\conj{\xi}}\alpha}^2+\norm{\alpha}_b^2).
	\end{multline*}
\end{lem}

\begin{proof}
	We first compute the difference between $\norm{\nabla''\alpha}^2$ and $\norm{\nabla'\alpha}^2$.
	By the divergence theorem,
	\begin{equation*}
		\norm{\nabla''\alpha}^2
		=-(\tr_g\nabla'\nabla''\alpha,\alpha)
		-\int_{M_\rho}\braket{\nabla_{\conj{\xi}}\alpha,\alpha}
	\end{equation*}
	and
	\begin{equation*}
		\norm{\nabla'\alpha}^2
		=-(\tr_g\nabla''\nabla'\alpha,\alpha)
		-\int_{M_\rho}\braket{\nabla_\xi\alpha,\alpha}.
	\end{equation*}
	The traces can be related to each other by
	\begin{equation*}
		\tr_g\nabla'\nabla''\alpha=\tr_g\nabla''\nabla'\alpha+\Ric^\circ\alpha+(\tr_gS)\alpha.
	\end{equation*}
	Hence
	\begin{equation}
		\label{eq:total-antiholomorphic-tangential-derivative}
		\begin{split}
			\norm{\nabla''\alpha}^2-\norm{\nabla'\alpha}^2
			&=-(\Ric^\circ\alpha,\alpha)-((\tr_gS)\alpha,\alpha)
			+\int_{M_\rho}\braket{\nabla_{\xi-\conj{\xi}}\alpha,\alpha}\\
			&=q(n+1)\norm{\alpha}^2-((\tr_gS)\alpha,\alpha)
			+\int_{M_\rho}\braket{\nabla_{\xi-\conj{\xi}}\alpha,\alpha}
			+o(\norm{\alpha}^2).
		\end{split}
	\end{equation}
	Next we compute the difference between $\norm{\nabla_{\conj{\xi}}\alpha}^2$ and $\norm{\nabla_\xi\alpha}^2$.
	Again by the divergence theorem,
	\begin{equation*}
		\norm{\nabla_{\conj{\xi}}\alpha}^2
		=-(\nabla_\xi\nabla_{\conj{\xi}}\alpha,\alpha)-((\div\xi)\nabla_{\conj{\xi}}\alpha,\alpha)
		-\int_{M_\rho}\braket{\nabla_{\conj{\xi}}\alpha,\alpha}
	\end{equation*}
	and
	\begin{equation*}
		\norm{\nabla_\xi\alpha}^2
		=-(\nabla_{\conj{\xi}}\nabla_\xi\alpha,\alpha)-((\div\conj{\xi})\nabla_\xi\alpha,\alpha)
		-\int_{M_\rho}\braket{\nabla_\xi\alpha,\alpha},
	\end{equation*}
	and hence
	\begin{multline}
		\label{eq:antiholomorphic-normal-derivative-pre}
		\norm{\nabla_{\conj{\xi}}\alpha}^2-\norm{\nabla_\xi\alpha}^2
		=-(R(\xi,\conj{\xi})\alpha+S(\xi,\conj{\xi})\alpha
		+\nabla_{[\xi,\conj{\xi}]}\alpha,\alpha)\\
		-((\div \xi)\nabla_{\conj{\xi}}\alpha,\alpha)
		+((\div \conj{\xi})\nabla_\xi\alpha,\alpha)
		+\int_{M_\rho}\braket{\nabla_{\xi-\conj{\xi}}\alpha,\alpha}.
	\end{multline}
	Now since $R(\xi,\conj{\xi})\alpha=-q\alpha-q\xi_\flat\wedge\iota_\xi\alpha+o(\abs{\alpha})$
	by \eqref{eq:asymptotic-curvature-behavior}, where $\xi_\flat$ is the metric dual of $\xi$, we get
	\begin{equation*}
		(R(\xi,\conj{\xi})\alpha,\alpha)=-q\norm{\alpha}^2-\norm{\iota_{\conj{\xi}}\alpha}^2+o(\norm{\alpha}^2).
	\end{equation*}
	On the other hand, by \eqref{eq:divergence-of-unit-normal} and \eqref{eq:bracket-of-unit-normal},
	\begin{equation*}
		\begin{split}
			&((\div\conj{\xi})\nabla_\xi\alpha,\alpha)-((\div\xi)\nabla_{\conj{\xi}}\alpha,\alpha)
			-(\nabla_{[\xi,\conj{\xi}]}\alpha,\alpha)\\
			&=(n-1)(\nabla_{\xi-\conj{\xi}}\alpha,\alpha)+(\nabla_{f\xi-\conj{f\xi}}\alpha,\alpha)
			+o(\norm{\alpha}^2+\norm{\nabla_b\alpha}^2),
		\end{split}
	\end{equation*}
	where $f$ is a smooth function defined near $\bdry\Omega$ that vanishes along $\bdry\Omega$.
	Moreover,
	\begin{equation*}
		\begin{split}
			(\nabla_\xi\alpha,\alpha)
			&=\int_{\mathcal{U}_\rho}\xi\abs{\alpha}^2-\conj{(\nabla_{\conj{\xi}}\alpha,\alpha)}
			=-\int_{\mathcal{U}_\rho}(\div\xi)\abs{\alpha}^2-\norm{\alpha}_b^2
			-\conj{(\nabla_{\conj{\xi}}\alpha,\alpha)}\\
			&=-n\norm{\alpha}^2-\norm{\alpha}_b^2-\conj{(\nabla_{\conj{\xi}}\alpha,\alpha)}
			+o(\norm{\alpha}^2)
		\end{split}
	\end{equation*}
	and similarly one gets
	$(\nabla_{f\xi}\alpha,\alpha)=o(\norm{\alpha}^2+\norm{\nabla_{\conj{\xi}}\alpha}^2+\norm{\alpha}_b^2)$.
	Therefore,
	\begin{multline*}
		\norm{\nabla_{\conj{\xi}}\alpha}^2
		-\norm{\nabla_\xi\alpha}^2
		=-(n^2-n-q)\norm{\alpha}^2+\norm{\iota_{\conj{\xi}}\alpha}^2
		-(S(\xi,\conj{\xi})\alpha,\alpha)
		-2(n-1)\re(\nabla_{\conj{\xi}}\alpha,\alpha)\\
		-(n-1)\norm{\alpha}_b^2
		+\int_{M_\rho}\braket{\nabla_{\xi-\conj{\xi}}\alpha,\alpha}
		+o(\norm{\alpha}^2+\norm{\nabla_b\alpha}^2+\norm{\nabla_{\conj{\xi}}\alpha}^2+\norm{\alpha}_b^2).
	\end{multline*}
	Combining this with \eqref{eq:total-antiholomorphic-tangential-derivative}, we obtain the lemma.
\end{proof}

Equation \eqref{eq:Morrey-Kohn-Hormander} and Lemma \ref{lem:Hormander-technique} imply
the following approximate equality for $\alpha\in\mathcal{D}^{0,q}(\overline{\mathcal{U}}_\rho;E)$:
\begin{multline}
	\label{eq:Modified-Andreotti-Vesentini}
	\norm{\partialbar\alpha}^2+\norm{\partialbar^*\alpha}^2
	=\norm{\nabla'_b\alpha}^2+\norm{\nabla_{\conj{\xi}}\alpha}^2
	+(n^2-n-q)\norm{\alpha}^2-\norm{\iota_{\conj{\xi}}\alpha}^2\\
	+(\mathring{S}\alpha-(\tr_g S)\alpha+S(\xi,\conj{\xi})\alpha,\alpha)
	-2(n-1)\re(\nabla_{\conj{\xi}}\alpha,\alpha)+(n-q-1)\norm{\alpha}_b^2\\
	+o(\norm{\alpha}^2+\norm{\nabla'_b\alpha}^2+\norm{\nabla_{\conj{\xi}}\alpha}^2+\norm{\alpha}_b^2).
\end{multline}
As a consequence, we obtain the following estimate.

\begin{prop}
	For $\alpha\in\mathcal{D}^{0,q}(\overline{\mathcal{U}}_\rho;E)$
	\begin{multline}
		\label{eq:final-general-estimate}
		\norm{\partialbar\alpha}^2+\norm{\partialbar^*\alpha}^2
		\ge(\mathring{S}\alpha-(\tr_gS)\alpha+S(\xi,\conj{\xi})\alpha,\alpha)
		+(n-q-2)\norm{\alpha}^2+(n-q-1)\norm{\alpha}_b^2\\
		+o(\norm{\alpha}^2+\norm{\alpha}_b^2)
	\end{multline}
	in the sense that, for any given $\varepsilon>0$, if $\rho>0$ is sufficiently small
	then inequality \eqref{eq:final-general-estimate}
	with $o(\norm{\alpha}^2+\norm{\alpha}_b^2)$ replaced by $-\varepsilon(\norm{\alpha}^2+\norm{\alpha}_b^2)$
	holds for any $\alpha\in\mathcal{D}^{0,q}(\overline{\mathcal{U}}_\rho;E)$.
\end{prop}

\begin{proof}
	By the Cauchy--Schwarz inequality, we have
	\begin{equation}
		\label{eq:Cauchy-Schwarz}
		(1-\varepsilon')\norm{\nabla_{\conj{\xi}}\alpha}^2-2(n-1)\re(\nabla_{\conj{\xi}}\alpha,\alpha)
		\ge -(1-\varepsilon')^{-1}(n-1)^2\norm{\alpha}^2.
	\end{equation}
	The proposition follows from \eqref{eq:Modified-Andreotti-Vesentini}, \eqref{eq:Cauchy-Schwarz}, and
	$\norm{\iota_{\conj{\xi}}\alpha}^2\le\norm{\alpha}^2$.
\end{proof}

We apply this proposition to $E=E_\delta=(T^{1,0},\varphi^{-\delta}g)$. Then, since
\begin{equation*}
	S=R-\delta\partial\conj{\partial}(-\log\varphi)\otimes I=R-\delta g\otimes I,
\end{equation*}
we get
\begin{align*}
	(\mathring{S}\alpha,\alpha)&\ge -(2q+q\delta)\norm{\alpha}^2+o(\norm{\alpha}^2),\\
	((\tr_gS)\alpha,\alpha)&=-(n+1+n\delta)\norm{\alpha}^2+o(\norm{\alpha}^2),\\
	(S(\xi,\conj{\xi})\alpha,\alpha)&\ge -(2+\delta)\norm{\alpha}^2+o(\norm{\alpha}^2).
\end{align*}
Therefore, for any $\varepsilon>0$, if $\rho>0$ is small enough then
\begin{equation}
	\label{eq:weighted-estimate}
	\norm{\partialbar\alpha}^2+\norm{\partialbar^*\alpha}^2
	\ge(2n-3q-3-\varepsilon)\norm{\alpha}^2+(n-q-1)\delta\norm{\alpha}^2
	+(n-q-1-\varepsilon)\norm{\alpha}_b^2
\end{equation}
for any $\alpha\in\mathcal{D}^{0,q}(\overline{\mathcal{U}}_\rho;E)$.
Thus Proposition \ref{prop:boundary-weighted-estimate} follows by Lemma \ref{lem:density},
and the proof of our main theorem is completed.
It is also obvious from \eqref{eq:weighted-estimate} that, if $n\ge 4$, then we actually do not need the weight
$\varphi^{-\delta}$.

\appendix
\section{On the vanishing result of Donnelly and Fefferman}
\label{sec:Donnelly-Fefferman}

Recall the following theorem on the space $\mathcal{H}_{(2)}^{p,q}(\Omega)$ of $L^2$ harmonic $(p,q)$-forms
due to Donnelly and Fefferman (which is restated in a way that is convenient for us).

\begin{thm}[Donnelly--Fefferman~\cite{Donnelly-Fefferman-83}, Donnelly~\cite{Donnelly-94}]
	\label{thm:Donnelly-Fefferman}
	Let $\Omega$ be a smoothly bounded strictly pseudoconvex domain of a Stein manifold $Y$
	equipped with the Cheng--Yau metric. Then,
	\begin{equation}
		\label{eq:Donnelly-Fefferman}
		\dim \mathcal{H}_{(2)}^{p,q}(\Omega)=
		\begin{cases}
			0, & p+q\not=n,\\
			\infty, & p+q=n.
		\end{cases}
	\end{equation}
\end{thm}

Actually, Donnelly--Fefferman~\cite{Donnelly-Fefferman-83} considered the case in which $Y=\mathbb{C}^n$,
and the metric was the Bergman metric (which is in fact quasi-equivalent to the Cheng--Yau metric).
For this case, Berndtsson~\cite{Berndtsson-96} has given another proof for $(p,q)=(n,1)$
in connection with the extension theorem of Ohsawa--Takegoshi.
The result of Donnelly~\cite{Donnelly-94} is more far-reaching: It applies to any complex manifold
$\Omega$ equipped with a complete K\"ahler metric $g$ whose associated 2-form $\omega$ admits the expression
$\omega=d\eta$ with a 1-form $\eta$ bounded with respect to $g$.
It is based on an observation of Gromov~\cite{Gromov-91}.

We shall see in this appendix that our technique provides another
proof of the vanishing part of Theorem \ref{thm:Donnelly-Fefferman}.
As we did in Section~\ref{sec:proof-main-theorem}, instead of the Cheng--Yau metric,
we can consider a metric $g$ of the form
\begin{equation*}
	\tensor{g}{_i_{\conj{j}}}=\partial_i\partial_{\conj{j}}(-\log\varphi),
\end{equation*}
where $\varphi$ is some smooth defining function of $\Omega$.
We apply the theory of geometric elliptic differential operators outlined in
Section~\ref{subsec:Einstein-deformations} to the Dolbeault Laplacian $\Delta_{\conj{\partial}}$.

By the Poincar\'e duality, it suffices to show $\mathcal{H}_{(2)}^{p,q}(\Omega)=0$ for $p+q<n$.
For these cases, on the complex hyperbolic space $\mathbb{C}H^n$, the coerciveness estimate
\begin{equation*}
	\norm{\alpha}^2\le C\norm{\Delta_\partialbar\alpha}^2,
	\qquad\alpha\in\dom\Delta_\partialbar\subset L^2\spacewedge^{p,q}(\Omega)
\end{equation*}
follows by the argument in~\cite{Donnelly-Fefferman-83}*{Section 3} based on
a formula of Donnelly--Xavier~\cite{Donnelly-Xavier-84}.
This makes Proposition \ref{prop:apriori} applicable. Our claim is the following.

\begin{lem}
	If $p+q<n$, then
	the indicial radius $R_{\Delta_{\conj{\partial}}}$
	of the Dolbeault Laplacian $\Delta_{\conj{\partial}}$ acting on
	$(p,q)$-forms is positive. Therefore, by Proposition \ref{prop:apriori},
	the space $\mathcal{H}_{(2)}^{p,q}(\Omega)$ is contained in $\mathcal{H}_{(2), \delta}^{p,q}(\Omega)$
	for some $\delta>0$.
\end{lem}

This lemma reduces the vanishing of $\mathcal{H}_{(2)}^{p,q}(\Omega)$ to that of the weighted cohomology
\begin{equation}
	\label{eq:weighted-Dolbeault-cohomology}
	H_{(2), \delta}^{p,q}(\Omega)=0
\end{equation}
by the same argument as in the last paragraph of Section \ref{subsec:reduction-to-L2-cohomology}.
Then, as follows, \eqref{eq:weighted-Dolbeault-cohomology} can be shown by
the Bochner--Kodaira--Nakano equality in the usual way.
Let $L_\delta$ be the trivial line bundle equipped with the fiber metric $\varphi^{-\delta}$.
Then the curvature $S$ of $L_\delta$ is given by
\begin{equation*}
	\tensor{S}{_i_{\conj{j}}}=-\delta\partial_i\partial_{\conj{j}}(-\log\varphi)
	=-\delta\tensor{g}{_i_{\conj{j}}}.
\end{equation*}
Therefore, for compactly supported smooth $(p,q)$-form $\alpha\in C^\infty_c\spacewedge^{p,q}(\Omega;L_\delta)$
with values in $L_\delta$, one has
\begin{equation*}
	\norm{\partialbar\alpha}^2+\norm{\partialbar^*\alpha}^2
	=\norm{D'\alpha}^2+\norm{(D')^*\alpha}^2+(n-p-q)\delta\norm{\alpha}^2,
\end{equation*}
where $D'$ is the holomorphic part of the covariant exterior derivative.
Since $C^\infty_c\spacewedge^{p,q}(\Omega;L_\delta)$ is dense in $\dom\partialbar\cap\dom\partialbar^*$,
this implies \eqref{eq:weighted-Dolbeault-cohomology}.

The computation of $R_{\Delta_{\conj{\partial}}}$ is tedious but straightforward.
The bundle of $(p,q)$-forms is associated to the representation
$\bigwedge^p\mathfrak{m}_0^*\otimes\bigwedge^q\conj{\mathfrak{m}}_0^*$ in the notation of
Section~\ref{subsec:Einstein-deformations} (where $\otimes$ denotes the tensor product over $\mathbb{C}$).
On the subspace $\bigwedge^p(\mathbb{C}^{n-1})^*\otimes_0\bigwedge^q(\conj{\mathbb{C}^{n-1}})^*$,
where $\otimes_0$ means that we take the totally trace-free part, the indicial roots $p+q$ and $2n-p-q$ appear.
These are the closest roots to the line $\re s=n$, which means that $R_{\Delta_{\conj{\partial}}}=n-p-q$.
The verification is left to the interested reader.

\bibliography{myrefs}

\end{document}